\documentclass[12pt]{amsart}

\pdfoutput=1

\usepackage[english]{babel}
\usepackage{textcomp}
\usepackage{amscd}
\usepackage{amssymb}
\usepackage{amsthm}
\usepackage{amsmath}
\usepackage{graphicx}
\usepackage[english]{babel}

\newtheorem{theorem}{Theorem}
\newtheorem{proposition}{Proposition}
\newtheorem{lemma}{Lemma}

\newtheorem{definition}{Definition}
\newtheorem{example}{Example}
\newtheorem{condition}{Condition}

\newcommand{\SH}{@{\hspace{2.4mm}}}  
\newcommand{\SHa}{@{\hspace{0.5mm}}}   

\newcommand{\degrees}{\ensuremath{^\circ}}
\newcommand{\Cpm}{{\mathbb{P}_\mu}} 
\newcommand{\Cpl}{{\mathbb{P}_\lambda}} 
\newcommand{\bP}{\mathbb{P}}
\newcommand{\tn}{{i_1 \dots i_n}} 
\newcommand{\tnj}{{j_1 \dots j_n}}
\newcommand{\tnks}{{\st{k}_n}}
\newcommand{\tnk}{{k_1 \dots k_n}}
\newcommand{\tns}{{\st{i}_n}}     
\newcommand{\tnjs}{{\st{j}_n}} 
\newcommand{\st}[1]{\underline{#1}} 
\newcommand{\T}[1][]{\cT_{#1}}
\newcommand{\cT}{\mathcal{T}}

\newcommand{\set}[1]{\left\{#1\right\}}
\newcommand{\setcard}[1]{\#{#1}} 
\newcommand{\exmat}[2][]{\mathcal{M}^{\scriptscriptstyle #1}\left({#2}\right)}
\newcommand{\evec}{\lvec{1}{1}}
\newcommand{\lvec}[2]{\begin{bmatrix}{#1}&{#2}\end{bmatrix}}

\newcommand{\Cprob}[1]{\prob[\Cpm]{#1}}
\newcommand{\Cprol}[1]{\prob[\Cpl]{#1}}
\newcommand{\prob}[2][\mathbb{P}]{{#1}\left({#2}\right)}

\newcommand{\A}{\mathbb{A}}

\newcommand{\E}{\mathbb{E}}

\newcommand{\ind}{\textbf{1}_}
\newcommand{\lp}{\phantom{\sigma^{k}}}   
\newcommand{\rp}{\phantom{...}}          

\begin{document}

\title{Correlated fractal percolation and the Palis conjecture}

\author{Michel Dekking}
\address{Michel Dekking, Delft Institute of Applied Mathematics,
 Technical University of Delft,  The Netherlands
\newline
\tt{F.M.Dekking{@}tudelft.nl}}

\author{Henk Don}
\address{Henk Don, Delft Institute of Applied Mathematics,
 Technical University of Delft,  The Netherlands
\tt{H.Don@tudelft.nl}}

\maketitle

\begin{abstract}

\noindent Let $F_1$ and $F_2$ be independent copies of correlated
fractal percolation, with Hausdorff dimensions $\dim_{\rm H}(F_1)$
and  $\dim_{\rm H}(F_2)$. Consider the following question: does
$\dim_{\rm H}(F_1)+\dim_{\rm H}(F_2)>1$  imply that their
algebraic difference $F_1-F_2$ will contain an interval? The well
known Palis conjecture states that `generically' this should be
true. Recent work by  Kuijvenhoven and the first author
(\cite{DK08}) on random Cantor sets can not answer this question
as their condition on the joint survival distributions of the
generating process is not satisfied by correlated fractal
percolation. We develop a new condition which permits us to solve
the problem, and we prove that the condition of  (\cite{DK08})
implies our condition. Independently of this we give a solution to
the critical case, yielding that a strong version of the Palis
conjecture holds for fractal percolation and correlated fractal
percolation: the algebraic difference contains an interval almost
surely if and only if the  sum of the Hausdorff dimensions of the
random Cantor sets exceeds one.
\end{abstract}

\section{Introduction}

In this paper we consider  a natural class (called correlated
fractal percolation) of random Cantor sets with dependence, as
opposed to the independent case, which is know as fractal
percolation or Mandelbrot percolation. Two and three dimensional
versions of both types of sets have occurred  before in the
literature, especially as a modeling tool, see e.g., \cite{SPB01},
where the dependent case is called the `homogeneous algorithm',
and the independent case the `heterogeneous algorithm' (See
Figure~\ref{fig:cfpofp} Left, respectively Right for an
illustration of these two processes by two realizations). In
\cite{M} they are called
 `constrained curdling', respectively `canonical curdling'.
All this work  has its roots in the seminal paper \cite{M74}.

Our main goal is to answer the question whether or not an interval
occurs in the algebraic difference of two independent random
Cantor sets from the correlated fractal percolation class. A
complete answer is given in Theorem \ref{corrfrac-class} in
Section~\ref{classify}.

We also call correlated fractal percolation $m$ out of $M$
percolation (cf. Subsection~\ref{sub:def-corr}), where $m$ is an
integer with $1\le m\le M$.
 It will appear that the transition
from no interval to interval lies at values of $m\approx
\sqrt{M}$. The combinatorial Lemma~\ref{Stperc} lies at the basis
for a solution of all cases, except the case $m=\sqrt{M+1}$, which
is a tough nut to crack (Lemma~\ref{Stperc2}).

\begin{figure}[!b]
\centering
 \includegraphics[width=13cm, angle=180]{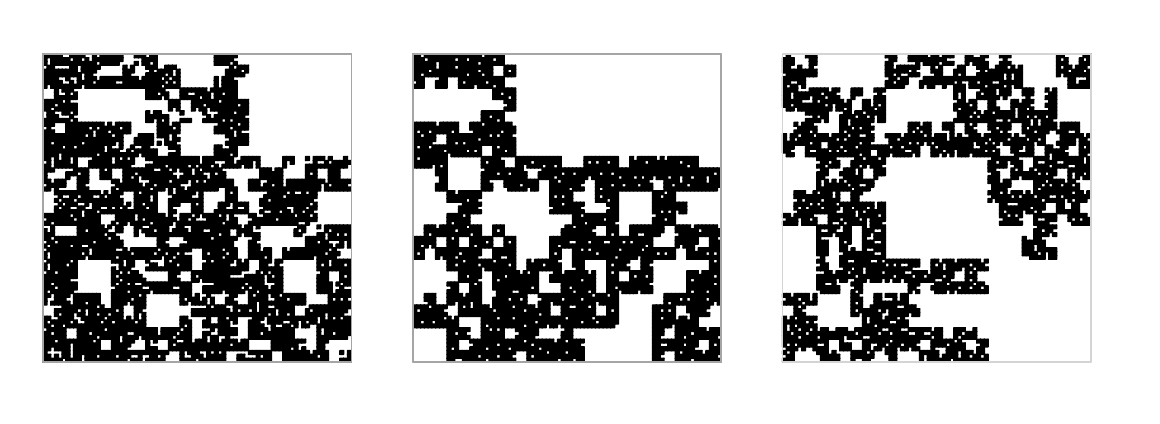}
\caption{\small Left: Two-dimensional 7 out of 9 correlated fractal percolation with $\mu(\emptyset)=0$.
 Middle: Two-dimensional 8 out of 9 correlated fractal percolation with $\mu(\emptyset)=\frac18$.
 Right: Ordinary two-dimensional fractal percolation with $p=7/9$.}
 \label{fig:cfpofp}
\end{figure}

 \section{Differences of random Cantor sets}

Here we will introduce $M$-adic random Cantor sets and their differences, and the main result
(Theorem~\ref{StDK08}) from \cite{DK08} regarding the Palis conjecture, with a rough sketch of the proof.
Finally we describe higher order Cantor sets which are particularly useful to obtain  a more complete characterization from Theorem~\ref{StDK08}.

\subsection{$M-$adic random Cantor sets}\label{sub:RCS}

An $M-$adic random Cantor set $F$ is constructed using the following
mechanism: take the unit interval and divide it into $M$ subintervals
of equal length. Each of those subintervals corresponds to a letter
in the alphabet $\mathbb{A} = \left\{0,\ldots,M-1\right\}$. It will
be convenient to consider $\A$ as an Abelian group with addition. So
for instance if $M=6$ we have $5 + 3 = 2$. Now define a \emph{joint
survival measure} $\mu$ on $2^{2^\mathbb{A}}$. It is determined by its values $(\mu(A))$ on the singletons $A\subset \A$. According to
this distribution we choose which subintervals are kept and which
are discarded. Then in each next construction step, each of the
surviving subintervals is again divided in $M$ subintervals of equal
length, of which a subset survives according to the distribution~$\mu$.

\smallskip

More formally, we consider the space of $\{0,1\}$-labeled $M$-adic trees  $\{0,1\}^\cT$,
 where we label each node $\tn\in\cT$ with $X_\tn\in\{0,1\}$.

The probability measure $\Cpm$ on this space
 is defined by requiring that $\Cprob{X_\emptyset=1}=1$ (where $\emptyset$ is the root of $\cT$), and that for all $\tn\in\cT$ the  random sets
\begin{align*}
  \set{i_{n+1}\in\A: X_{\tn i_{n+1}} = 1} \end{align*}
are independent and identically distributed according to $\mu$.
We let $\T[n]$ denote the set of nodes at level $n$, and for any $\tns=\tn$ from $\T[n]$ we define the
associated $M$-adic interval by
\begin{equation*}
I_{\tn}:=\left[\frac{i_1}{M}+\cdots+\frac{i_{n-1}}{M^{n-1}}+
 \frac{i_n}{M^n},\frac{i_1}{M}+\cdots+\frac{i_{n-1}}{M^{n-1}}+\frac{i_n+1}{M^n} \right].
\end{equation*}
The $n$-th level approximation $F^n$ of the random Cantor set is a
union of such $n$-th level $M$-adic intervals selected by the sets
$S_n$ defined by
$$ S_n=\{\tn: X_{i_1}=X_{i_1i_2}=\dots=X_{\tn}=1\}.$$
The random Cantor set $F$ is
$$F=\bigcap_{n=1}^{\infty } F^n=\bigcap_{n=1}^{\infty } \bigcup_{\;\;\tn \in S_n}I_{\tn}. $$
The \emph{marginal probabilities} $p_i$ of $\mu$  are defined for $i\in \mathbb{A}$
by
\begin{equation}\label{def:marg}
p_i := \sum_{X\subseteq \mathbb{A}:i\in X}\mu(X).
\end{equation}
\\
We start with the definition of the class of random Cantor sets
which we will take into consideration.

\subsection{Correlated fractal percolation}\label{sub:def-corr}

From now on we will consider one-dimensional fractal percolation.

\begin{definition}
Suppose $\mu$
assigns the \emph{same} positive probability to all subsets of
$\A$ with $m$ elements for some fixed integer $1\le m \le M$, and that $\mu$ assigns probability zero to all other non-empty subsets of $\A$. If $p:=(1-\mu({\emptyset}))\frac{m}{M}$ then we call this
$(m,M,p)$-percolation.
\end{definition}

\smallskip

 We can compute the marginal probabilities of $(m,M,p)$-percolation as follows.
Let $X$ be a subset of $\A$, chosen
according to the joint survival distribution $\mu$.
 The probability that
$X$ is non-empty is $1-\mu(\emptyset)$. Given that $X$ is
non-empty, the probability that a fixed $k\in\A$ belongs to $X$
equals $m/M$. It follows that for $k\in\A$ the
marginal probability $p_k$ is given by
\begin{equation*}
p_k = (1-\mu(\emptyset))\frac{m}{M} = p,
\end{equation*}
which is exactly the reason why we defined $(m,M,p)$-percolation by
requiring that $p=(1-\mu({\emptyset}))m/M$. Because
$0\leq\mu(\emptyset)\leq 1$, $(m,M,p)$-percolation is only defined
for $0\leq p\leq \frac{m}{M}$. From now on we will assume that $p>0$
and $m>0$, since giving the empty set probability one does not yield
the most exciting situation.

\subsection{Algebraic differences of sets}

The \emph{algebraic difference} $F_1-F_2$ of the sets $F_1$ and $F_2$ is defined by
$$F_1-F_2=\{x-y:\; x\in F_1,\; y\in F_2\}.$$

The well known Palis conjecture (\cite{PT}) states that
`generically' $\dim_{\rm H}(F_1)+\dim_{\rm H}(F_2)>1$ should imply
that the algebraic difference $F_1-F_2$ will contain an interval.

This question is considered in \cite{DS08} and \cite{DK08} for two
$M$-adic random Cantor sets $F_1$ and $F_2$ with the same $M$ but
not necessarily the same joint survival distribution.\\ One can
distinguish between joint survival distributions selecting
intervals independently and joint survival distributions not
having this property. In the independent case, the problem is
somewhat less complicated, but still far from trivial. Intervals
are selected and discarded independently if and only if the joint
survival distribution satisfies for all $X\subseteq \A$ the
equality
\begin{equation}\label{eq:ind}
\mu(X)=\prod_{i\in X}p_i\prod_{i\not\in X}(1-p_i).
\end{equation}

An important role in the answer to the main question is played by
 the \emph{cyclic cross-correlation coefficients} (mostly simply called correlation coefficients)
\begin{align*} \label{gamma_k def}
 \gamma_{k} &:= \sum_{i=0}^{M-1} q_ip_{i+k},\quad {\rm for}\; k\in\A,
\end{align*}
where $(p_i)$ and $(q_i)$ are the vectors of marginal
probabilities of the joint survival distributions $\mu$,
respectively $\lambda$.
\ \\
\ \\
The result of \cite{DK08} needs the following condition (which is satisfied  in the independent case of Equation~(\ref{eq:ind})).
\begin{condition}\label{JSC}
A joint survival distribution $(\mu(A))_{A\subseteq \A}$ satisfies the \emph{joint survival condition} {\rm (JSC)} if it assigns
positive probability to the marginal support $\emph{Supp}_m(\mu)$ of $\mu$,
which is defined by
\begin{equation*}
\emph{Supp}_m(\mu):=\bigcup \left\{X\subseteq
\mathbb{A}:\mu(X)>0\right\}=\{i\in \A: p_i>0\}.
\end{equation*}
\end{condition}

The following result of \cite{DK08} generalizes the main theorem of \cite{DS08}.
\begin{theorem}\label{StDK08}
Consider two independent random Cantor sets $F_1$ and $F_2$ whose
joint survival distributions $\mu$ and $\lambda$ both satisfy Condition
\ref{JSC}, the {\rm (JSC)}.
\begin{enumerate}
\item If $\gamma_k >1$ for all $k \in\A$, then $F_1-F_2$
contains an interval a.s. on $\left\{F_1-F_2 \not= \emptyset
\right\}$. \item If $\gamma_k<1, \gamma_{k+1}<1$ for some $k \in\A$,
then $F_1-F_2$ contains no interval a.s.
\end{enumerate}
\end{theorem}

Obviously for $(m,M,p)$-percolation the JSC
is not satisfied, unless we are in the case $m=M$, giving positive
probability only to the full alphabet and the empty set (actually,
this is ordinary fractal percolation, where intervals
are discarded independently and the marginal probabilities $p_k$
are all equal to $p$).

\subsection{The geometry of the algebraic difference}

We will give in this subsection the tools and the notation introduced in \cite{DS08} and \cite{DK08}.

Let $\phi:[0,1]^2\to[-1,1]$ be given by $\phi(x,y)=y-x$, then $F_1-F_2=\phi(F_1 \times F_2)$.
Thus $F_1-F_2$ is  defined on the product space  of the probability spaces of $F_1$ and $F_2$.
We will use $\bP:=\Cpm\times \Cpl$ to denote the corresponding product measure
 and $\E$ to denote expectations with respect to this probability.
\\
\\
Let $F_1$ and $F_2$ be two independent $M$-adic random Cantor sets with
 joint survival distributions $\mu$ and $\lambda$, respectively.
Denote by $F_1^n$ and $F_2^n$ their $n^{\rm th}$ level approximations ($n\ge0$) and
 define the following subsets of the unit square $[0,1]^2$:
\begin{align*}
  \Lambda^n &:= F_1^n \times F_2^n, \quad n\ge0, &
  \Lambda   &:= F_1 \times F_2 = \bigcap_{n=0}^\infty \Lambda^n.
\end{align*}
Note that as $F_1^n \downarrow F_1$ and $F_2^n \downarrow F_2$, also $\Lambda^n \downarrow \Lambda$.

\begin{figure}[t!]
\centering
\includegraphics*[width =9.8cm]{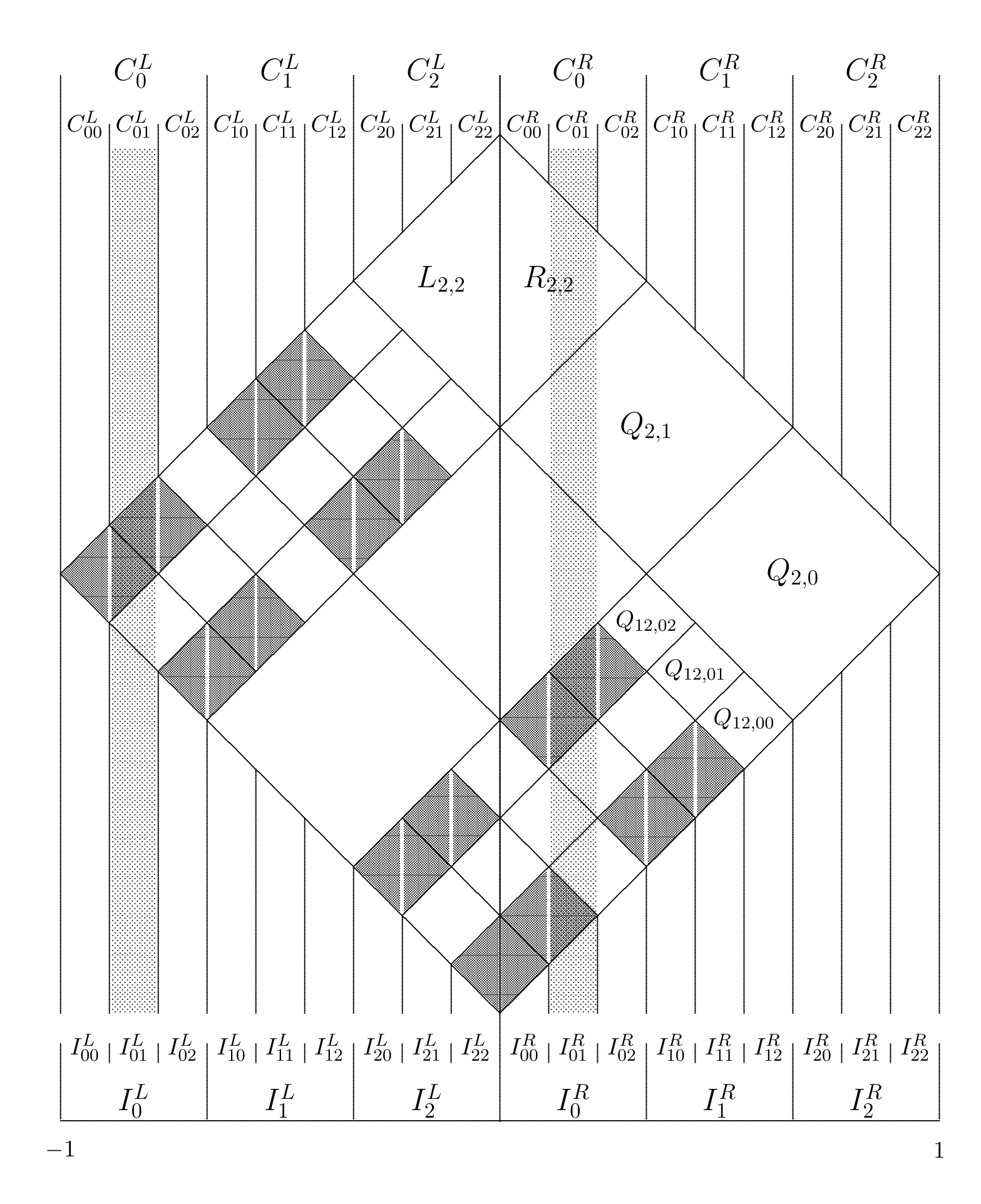}
\caption{\small
 An illustration for $M=3$ of the unit square $[0,1]^2$, scaled and rotated by 45\degrees{}.
 The shaded squares form a realization of $\Lambda^2$ for 2 out of 3 fractal percolation. The vertical  projection gives the $\phi$-image $[-1,\,5/9]$ of $\Lambda^2$.}%
\label{projection}%
\end{figure}

The $\Lambda^n$ are unions of $M$-adic squares
\begin{align*}
  Q_{\tn,\tnj} := I_{\tn} \times I_{\tnj},
\end{align*}
with $\tn,\tnj\in\T[n]$ and $n\ge0$.

 Note that $\phi$ acts as a 45\degrees{} projection on the $x$-axis. Similarly to \cite{DS08} and \cite{DK08} we scale and rotate the unit square over  45\degrees{} counterclockwise, to rather see it as a 90\degrees{} projection on $[-1,1]$. See Figure~\ref{projection} for a graphical representation of some of the squares $Q$ and their $\phi$-images.
 Here we denote the $M$-adic intervals $I_{\tn}$ in $[0,1]$ by $I_{\tn}^R$ (they are projections of squares in the right side of the tilted  square), and define
 $$I_{\tn}^L=I_{\tn}^R\!-1,$$
 for the $M$-adic intervals $I_{\tn}$ in $[-1,0]$ (they come from the left side). The \emph{columns} $C^U_\tnk$, where $U=L$ or $U=R$ are defined for each $\tnk\in\cT$ by
 $$ C^U_\tnk :=\phi^{-1}\left(I_\tnk^U\right).$$
 Note that any $n^{\rm th}$ level $M$-adic square $Q_{\tn,\tnj}$
 is split into a `left' and a `right' triangle by the $M$-adic columns.
These triangles are called $L$-triangles and $R$-triangles, and will be denoted by
$L_{\tn,\tnj}$ and $R_{\tn,\tnj}$ respectively, for any $\tn,\tnj\in\T$.
\\
\\
For all $U,V\in\set{L,R}$ and $\tnks\in\T$ we let
\begin{align*}
    Z^{UV}(\tnks)
 := \setcard{\set{\big(\tns,\tnjs\big): Q_{\tns,\tnjs}\subseteq\Lambda^n, V_{\tns,\tnjs}\subseteq C^U_{\tnks}}}
\end{align*}
denote the number of level~$n$ $V$-triangles in $\Lambda^n \cap C^U_\tnks$.
We also denote the total number of $V$-triangles in columns $C^L_{\tnks}$ and $C^R_{\tnks}$ together by
\begin{align*}
  Z^V(\tnks) := Z^{LV}(\tnks) + Z^{RV}(\tnks),
\end{align*}
for all $\tnks\in\T$. For example, in Figure~\ref{projection} we have  $Z^R(01)=1+2=3$.
\\
\\
An important observation is that an $M$-adic interval $I^U_{\tnks}$ is absent in
  $\phi(\Lambda^n)$ exactly when
 there are no triangles  in the corresponding column $C^U_{\tnks}$ in $\Lambda^n$:
\begin{align*}
  I^U_\tnks \not\subseteq \phi(\Lambda^n) \iff Z^{UL}(\tnks)=Z^{UR}(\tnks)=0.
\end{align*}

The triangle counts $Z^{UV}(\tnks)$, with $k_1,k_2,\dots$ a fixed path,
 constitute a two type branching process in a varying environment with interaction:
  the interaction comes from the dependency between triangles that are \emph{aligned},
 i.e., triangles contained in respective squares $Q_{\tn,\tnj}$ and $Q_{i_1'\dots i_n',j_1'\dots j_n'}$
 with $\tn=i_1'\dots i_n'$ or $\tnj=j_1'\dots j_n'$.
Squares that are not aligned will be called \emph{unaligned}.
\\
\\
 The \emph{expectation matrices} of the two type branching process are for $\tnks\in\T$ given by:
\begin{align} \label{exmat def}
 \exmat{\tnks} := \begin{bmatrix}
   \E Z^{LL}(\tnks) & \E Z^{LR}(\tnks) \\
   \E Z^{RL}(\tnks) & \E Z^{RR}(\tnks)
 \end{bmatrix}.
\end{align}
These matrices satisfy the basic relation
\begin{equation} \label{exmat expansion}
 \exmat{\tnk} = \exmat{k_1} \cdots \exmat{k_n},
\end{equation}
for all $\tnk\in\T$.

Lemma \ref{exmat sum to gamma} shows the importance of the correlation coefficients.


\begin{lemma} \label{exmat sum to gamma}{\rm (\cite{DS08})}
For all $k\in\A$ we have
\begin{align}\label{columnsums}
   \evec\exmat{k} = \lvec{\E Z^L(k)}{\E Z^R(k)} =\big[\gamma_{k+1}\; \gamma_{k}\big].
\end{align}
\end{lemma}

\begin{proof}
As in \cite{DS08} this follows from some careful bookkeeping and
$$ \bP{(Q_{i,j} \!\subseteq \Lambda^1)}
  = \bP{(I_i \!\subseteq F^1_1, I_j \!\subseteq F^1_2)}
  = \Cprob{I_i \!\subseteq F^1_1}\!\Cprol{I_j \!\subseteq F^1_2}
  = p_iq_j. $$
\end{proof}

\subsection{Rough sketch of the proof of Theorem~\ref{StDK08}}

The idea of the proof is to
pair unaligned left and right triangles
 that survive in the \emph{same} column into what are called \emph{$\Delta$-pairs}.

Suppose we have a $\Delta$-pair in one of the columns with positive
probability. If we can prove that there is a strictly positive
probability that the number of $L$-triangles and $R$-triangles in
\emph{all subcolumns} of this column grows exponentially, then it
can be shown that with positive probability the $M$-adic interval
corresponding to this column is in the projection $\phi(\Lambda)$.
The determining quantity for exponential growth is the smallest correlation coefficient
\begin{align} \label{gamma def}
 \gamma &:= \min_{k\in\A} \gamma_{k}.
\end{align}

Now we make use of the fact that conditioned on
$\Lambda\not=\emptyset$ the Hausdorff dimension of $\Lambda$ is
almost surely larger than 1, which is implied by $\gamma>1$.

 It can
be shown (see \cite{DS08}) that from this it follows that the number
of unaligned squares grows to infinity. By self-similarity of the
process each of the unaligned squares has positive probability to
generate an interval in the projection, and hence with probability
one there will be an interval in the projection.
\\
\\
To show that a $\Delta$-pair occurs somewhere with positive
probability it suffices that $\gamma>1$. So the joint survival
condition is only needed to ensure positive probability of
exponential growth in all subcolumns of a $\Delta$-pair. For any
level $l$ $\Delta$-pair $(L^{l},R^{l})$ that is contained in a level $l$
column $C$, the distribution of the number of level $l+n$
$V$-triangles surviving in $\Lambda^{l+n}$ in the
$\underline{k}_n$-th subcolumn of $(L^{l},R^{l})$, conditional on the
survival of $(L^{l},R^{l})$ in $\Lambda^{l}$, is independent of $l$, the
particular choice of the column $C$ and the $\Delta$-pair in this column.
Therefore, we can unambigiously denote a random variable having this
distribution by
\begin{equation}
\tilde Z^V(\underline{k}_n)
\end{equation}
for all $V\in\left\{L,R\right\}$ and
$\underline{k}_n\in\mathcal{T}$. In general $\tilde
Z^V(\underline{k}_n)$ does not have the distribution of
$Z^V(\underline{k}_n)$ because there is possible dependence
between the offspring generation of two level $0$ triangles,
whereas there is no dependence between the offspring generation of
the $L$-triangle and the $R$-triangle of a $\Delta$-pair, because
they are unaligned by definition of a $\Delta$-pair. However, both
do have the same expected value.

In \cite{DK08} the following lemma on exponential
growth of triangles is proved:
\begin{lemma}\label{DK:lemma_1}
If $\gamma>1$, and the joint survival distributions satisfy the joint
survival condition, then for all $n\geq 0$
\begin{equation*}
\mathbb{P}(\tilde Z^L(\underline{k}_{l})\geq \gamma^l , \tilde Z^R(\underline{k}_{l})\geq \gamma^l \emph{\ for\ all\ } \underline{k}_l\in \cT_{l} \emph{\ for\ all\ } 0\leq
l\leq n)>0.
\end{equation*}
\end{lemma}
In Lemma~\ref{lemma_expgrowth} in Section~\ref{sec:DGC} we obtain this lemma (with a
different growth factor) under weaker conditions than the joint
survival condition.

\subsection{Higher order Cantor sets}\label{sub:high}

The idea of higher order Cantor sets is to collapse $n$ construction steps into one step. Since $\Lambda^n \downarrow \Lambda$  we can for all $n\geq 1$ write
$$
\Lambda = \bigcap_{m=1}^\infty \Lambda^m = \bigcap_{m=1}^\infty \Lambda^{nm}.
$$
The sets $(\Lambda^{nm})_{m=1}^\infty$ are constructed by joint survival distributions which will be denoted by $\mu^{(n)}$ and $\lambda^{(n)}$. If Theorem \ref{StDD09} fails to answer the interval or not question for the pair $(\mu,\lambda)$, one can hope to get an answer by considering $\Lambda$ as generated by $(\mu^{(n)},\lambda^{(n)})$.

The success of this idea is illustrated by Theorem 6.1 in \cite{DK08}, and by Theorem~\ref{St:spectral}. We will also use it for the proof of Lemma~\ref{Stperc2}.

All entities of the $n$th order random Cantor set will be denoted with a superscript $(n)$. The alphabet now is $\A^{(n)}=\left\{0,\ldots,M^n-1\right\}$ and $\mu^{(n)}$ and $\lambda^{(n)}$ are probability measures on the subsets of $\A^{(n)}$ which are completely determined by $\mu$ and $\lambda$.
\\
\\
Let us illustrate this with a simple example. Let $M=2$ and define $\mu$ by $\mu(\left\{0,1\right\})=\mu(\left\{1\right\})=1/2$. For the corresponding second order Cantor set we have $\A^{(2)}=\left\{0,1,2,3\right\}$ and
$$
\begin{array}{c}
\mu^{(2)}(\left\{0,1,2,3\right\}) = \mu^{(2)}(\left\{1,2,3\right\}) = \mu^{(2)}(\left\{0,1,3\right\}) = \mu^{(2)}(\left\{1,3\right\}) = \frac{1}{8},\\
\mu^{(2)}(\left\{2,3\right\}) = \mu^{(2)}(\left\{3\right\}) = \frac{1}{4}.
\end{array}
$$

\section{The critical case}

What happens in the critical case when $\gamma=1$? This was left
open in \cite{DS08} and \cite{DK08}. Here we will give a simple
argument, independent of the other results in this paper, that
permits us to give a complete classification in
Theorem~\ref{corrfrac-class}. In particular we can tell what
happens for critical classical fractal percolation: if
$p=1/\sqrt{M}$, then there is almost surely no interval in the
difference set.

\begin{proposition}\label{crit}
Consider two independent random Cantor sets $F_1$ and $F_2$ with
joint survival distributions $\mu$ and $\lambda$ having marginal
probabilities $(p_i)$ and $(q_j)$, such that $\gamma_0\le 1$. Then
$F_1-F_2$ contains no interval a.s., provided that for all $
i\in\A:p_iq_i\ne 1$.
\end{proposition}

\textbf{Proof.}
Let $Z_n$ be the number of `central' squares in $\Lambda^n$, i.e.,
$$Z_n=\#\{\tn\in \cT: Q_{\tn,\tn}\in\Lambda^n\}.$$
Then $Z_0=1$, and since these central squares are unaligned,
$(Z_n)$ is an ordinary branching process with mean offspring
$$\E[Z_1]=p_0q_{0}+p_1q_{1}+...+p_{M-1}q_{M-1}=\gamma_0\le 1.$$
Now if $\gamma_0=1$, then  the offspring distribution is
deterministic  ($Z_1\equiv 1$) if and only if  $p_i=q_i=1$ for
some $i\in  \A$, which is assumed \emph{not} to be the case.
Hence, $(Z_n)$ will die out a.s., say at time $N$. In the sequel
we will write the string $\tn=(k,k,\dots, k)$ for $k \in \A$ as
$k^n$.
\\
Then,  because there are no central squares left, $C^R_{0^{N+n}}$
only contains left triangles for all $n\ge 0$. Moreover, the
number of left triangles in $(C^R_{0^{N+n}})$ is an ordinary
branching process $(Y_n^R)$ with random initial distribution
$Y_0^R$, and mean offspring
$$\E[Y_1^R]=p_0q_{M-1}\le 1.$$
Similarly, $C^L_{(M-1)^{N+n}}$ only contains right triangles for
all $n\ge 0$. Moreover, the number of right triangles in
$(C^L_{(M-1)^{N+n}})$ is a branching process $(Y_n^L)$ with
$Y_0^L$, and mean offspring
$$\E[Y_1^L]=p_{M-1}q_{0}\le 1.$$
If both $\E[Y_1^R]$ and $\E[Y_1^L]$ would equal $1$, then
$p_0q_{M-1}=p_{M-1}q_{0}=1$ and consequently
$p_0q_0=p_{M-1}q_{M-1}=1$ implying that $\gamma_0\geq 2$. Hence
either $\E[Y_1^R]<1$ or $\E[Y_1^L]<1$, such that at least one of
the two branching processes $(Y_n^R)$ and $(Y_n^L)$ will die out
almost surely, implying that $F_1-F_2$ has a `gap' directly left
or right of $0$. It then follows from selfsimilarity and the
denseness of the points $k_1M^{-1}+\dots+k_nM^{-n}$ that $F_1-F_2$
contains no interval a.s. (cf. \cite{DS08})\hfill $\Box$
\\
\\
That we need at least some restriction on the marginal
probabilities in addition to the requirement $\gamma_0\le 1$ is
shown in the following example: Let $M=2$ and define the
(deterministic) joint survival distributions $\mu$ and $\lambda$
by setting $\mu(\left\{0\right\})=1$ and
$\lambda(\left\{0,1\right\})=1$. Then $F_1\times F_2 =
\left\{0\right\}\times [0,1]$, and so  $F_1-F_2=[-1,0]$.

\section{The distributed growth condition}\label{sec:DGC}

In this section we introduce a condition for exponential growth of
triangles which is based on the following idea: if we can find a
column $C$ where we have a sufficient number of $\Delta$-pairs,
then under some conditions each of these $\Delta$-pairs can be
used to guarantee exponential growth of triangles in a proper
subset of the set of subcolumns of $C$. In some sense we `spread
the burden of proof', and this gives the condition a flexible
nature. This is illustrated by the fact that with help of this
condition, we can completely classify correlated fractal
percolation.
\\
\\
For $X,Y\subseteq\A$ and $e\in\A$ we define $\gamma_e(X,Y)$ to be
the $e^{\rm th}$ correlation coefficient corresponding to the
joint survival distributions $\mu^\star$ and $\lambda^\star$
assigning probability one to $X$ and $Y$ respectively, i.e.,
\begin{equation}\label{gammak}
\gamma_e(X,Y) = \sum_{i\in \A} \ind{Y}(i)\ind{X}(i+e).
\end{equation}

\begin{condition}\label{DGC} The pair of joint survival distributions $(\mu,\lambda)$
satisfies the \emph{distributed growth condition} if for all $k\in\A$
we can find sets $X_k,Y_k\subseteq\A$  such that
\begin{itemize}
\item[(DG0)]\quad $\mu(X_k)>0$ and $\lambda(Y_k)>0$, \\[-.35cm]
\item[(DG1)]\quad $\displaystyle\min_{e\in \A} \gamma_e(X_k,Y_k) \geq 1$,
\item[(DG2)]\quad $\gamma_k(X_k,Y_k)\geq 2,\; \gamma_{k+1}(X_k,Y_k)\geq 2$.\\[-.35cm]
\end{itemize}
\end{condition}

\begin{lemma}\label{lemma_DG3}
Let $E$ denote the event that there exists $l\geq 1$,
$\underline{k}_l\in\mathcal{T}_l$ and $U\in\left\{L,R\right\}$
such that $C_{\underline{k}_l}^U$ contains at least $M$ left and
$M$ right triangles which are all pairwise unaligned. If the pair
of joint survival distributions $(\mu,\lambda)$ satisfies the DGC,
then
$$
\bP(E)>0.
$$
\end{lemma}

\textbf{Proof.} Choose $X_0,Y_0\subseteq \A$ according to the DGC.
Define the joint survival distributions $\mu^\star$ and
$\lambda^\star$ by $\mu^\star(X_0)=\lambda^\star(Y_0)=1$. Then by
(DG2) both column sums of the expectation matrix
$\mathcal{M}^\star(0)$ are at least $2$, implying that
$$
[1\quad 1]\,\mathcal{M}^\star(0^n) \geq [2^n\quad 2^n],
$$
elementwise. The first row of $\mathcal{M}^\star(0^n)$ corresponds
to $C^L_{0^n}$, which can contain at most one left triangle and no
right triangles. Therefore, both numbers in the second row of
$\mathcal{M}^\star(0^n)$ are bounded below by $2^n-1$. It follows
that the numbers of left and right triangles in $C^R_{0^n}$ grow
arbitrary large if $n$ is sufficiently large. Since $\mu$ and
$\lambda$ assign positive probability to $X_0$ and $Y_0$
respectively, the statement of the lemma follows.\hfill $\Box$
\\
\\
We can now formulate our exponential growth lemma.

\begin{lemma}\label{lemma_expgrowth}
If the pair of joint survival distributions $(\mu,\lambda)$
satisfies the distributed growth condition, then there exist
$l\geq 1$, $\underline{k}_l\in\mathcal{T}_l$ and $\eta>1$ such
that for all $n\geq 0$
$$
\mathbb{P}(Z^L(\underline{k}_l\underline{k}_p)\geq
\eta^p,Z^R(\underline{k}_l\underline{k}_p)\geq \eta^p \emph{\ for\
all\ } \underline{k}_p\in\mathcal{T}_p \emph{\ for\ all\ } 0\leq
p\leq n)>0.
$$
\end{lemma}

\textbf{Proof.} Choose $n\geq 0$ arbitrary. For all $k\in\A$
choose $X_k\subseteq\A$ and $Y_k\subseteq\A$ such that these sets
satisfy the DGC. Define the joint survival distributions
$\mu_k^\star$ and $\lambda_k^\star$ by requiring that
$\mu_k^\star(X_k)=\lambda_k^\star(Y_k)=1$.

Let $k\in\A$ be fixed and consider the expectation matrices
corresponding to the triangle growth process defined by
$(\mu_k^\star$,$\lambda_k^\star)$. By (\ref{columnsums}), their
column sums are given by the correlation coefficients
corresponding to the pair of joint survival distributions
$(\mu_k^\star,\lambda_k^\star)$. So, for all $e\in\A$, both column
sums of $\mathcal{M}_k^\star(e)$ are at least $1$ and both column
sums of $\mathcal{M}_k^\star(k)$ are at least $2$. Let $p$ be an
integer with $0\leq p\leq n$. Since for $\underline{k}_p =
k_1\ldots k_p \in\mathcal{T}_p$ we have
\begin{equation*}
\mathcal{M}_k^\star(\underline{k}_p) =
\mathcal{M}_k^\star(k_1)\ldots\mathcal{M}_k^\star(k_p),
\end{equation*}
it follows that a lower bound for the column sums of
$\mathcal{M}_k^\star(\underline{k}_p)$ is determined by the number
of $k$'s in the string $\underline{k}_p$. We obtain (omitting the
dependence on $k$, and writing $k_j$ for the $j$th element in the
string $\underline{k}_p$.)
\begin{equation*}
\gamma^\star_{\underline{k}_p} \geq 2^{\#\left\{0\le j\le p: k_j=k\right\}},\quad
\gamma^\star_{\underline{k}_p+1} \geq 2^{\#\left\{0\le j\le p: k_j=k\right\}}.
\end{equation*}
From the deterministic nature of $\mu_k^\star$ and
$\lambda_k^\star$, it follows that the expectation of the number
of triangles in some column is simply the number that will occur.
This means that for all $0\leq p\leq n$
\begin{eqnarray*}
Z_k^{L;\star}(\underline{k}_p) &=& \E[Z_k^{L;\star}(\underline{k}_p)] = \gamma^\star_{\underline{k}_p+1} \geq 2^{\#\left\{0\le j\le p: k_j=k\right\}},\\
Z_k^{R;\star}(\underline{k}_p) &=& \E[Z_k^{R;\star}(\underline{k}_p)] = \gamma^\star_{\underline{k}_p} \geq 2^{\#\left\{0\le j\le p: k_j=k\right\}}.
\end{eqnarray*}
Since $(\mu,\lambda)$ satisfies the DGC, we can by Lemma
\ref{lemma_DG3} find an $l$-adic column $C_{\underline{k}_l}^U$
containing with strictly positive probability at least $M$ left-
and $M$ right triangles being all pairwise unaligned. Let this
event be denoted by $E$ and abbreviate the notation of this column
by $C$ and its subcolumns $C_{\underline{k}_l\underline{k}_p}^U$
by $C_{\underline{k}_p}$.
\\
\\
Now suppose we have a $\Delta$-pair $(L,R)$ in $C$, in which the
growth process behaves according to the pair of joint survival
distributions $(\mu_k^\star,\lambda_k^\star)$. Then, for all $p$
and all subcolumns $C_{\underline{k}_p}$ of $C$, both the number
of left and the number of right triangles in
$C_{\underline{k}_p}\cap (L\cup R)$ is at least $2^{\#\left\{0\le
j\le p:k_j=k\right\}}$.
\\
\\
Conditional on the event $E$, we have $M$ left and right triangles
in $C$. We can label them by the elements of $\A$ such that we
have $M$ $\Delta$-pairs. These $2M$ triangles are all pairwise
unaligned (also if they belong to different $\Delta$-pairs) and
hence there is completely no dependence between these triangles.
It follows that it is possible that in each of the $\Delta$-pairs
the growth process takes place as prescribed by $\mu_k^\star$ and
$\lambda_k^\star$, where $k$ is the label of the $\Delta$-pair.
Denoting the event that this happens in the first $n$ construction
steps after occurrence of $E$ by $E_n$, we can find a strictly
positive lower bound for $\mathbb{P}(E_n|E)$:
\begin{equation*}
\bP(E_n|E) \geq \prod_{k\in\A} \mu(X_k)^{\sum_{j=1}^n (\#
X_k)^{j-1}} \lambda(Y_k)^{\sum_{j=1}^n (\# Y_k)^{j-1}} > 0.
\end{equation*}
Let $0\leq p\leq n$ and let $C_{\underline{k}_p}$ be an arbitrary
$M^p$-adic subcolumn of $C$. There must exist a
$k=k(\underline{k}_p)\in\A$ such that $\#\left\{0\le j\le
p:k_j=k\right\} \geq \lceil \frac{p}{M}\rceil$. Hence, given the
event $E_n$, for the numbers of left and right triangles in
$C_{\underline{k}_p}$ we have
\begin{equation*}
Z^L(\underline{k}_l\underline{k}_p)\geq 2^{\lceil\frac{p}{M}\rceil},\quad
Z^R(\underline{k}_l\underline{k}_p)\geq 2^{\lceil\frac{p}{M}\rceil}.
\end{equation*}
Taking $\eta = \sqrt[M]{2}$, we obtain
\begin{eqnarray*}
\quad\mathbb{P}(Z^L(\underline{k}_l\underline{k}_p)\geq \eta^p,Z^R(\underline{k}_l\underline{k}_p)\geq \eta^p \mbox{\ for\ all\ }\underline{k}_p\in\mathcal{T}_p \mbox{\ for\ all\ }0\leq p\leq n)&&\\
\geq \mathbb{P}(E)\mathbb{P}(E_n|E)>0.\hspace{7cm}&&
\end{eqnarray*}
\hfill $\Box$
\\
\\
Collecting the results established so far, we can replace the
joint survival condition (Condition \ref{JSC}) and Lemma
\ref{DK:lemma_1} by the distributed growth condition and Lemma
\ref{lemma_expgrowth} to obtain the following useful variation on
Theorem \ref{StDK08}:

\begin{theorem}\label{StDD09}
Consider two independent random Cantor sets $F_1$ and $F_2$ whose
joint survival distributions satisfy Condition \ref{DGC}, the
{\rm DGC}.
\begin{enumerate}
\item If $\gamma_k >1$ for all $k \in \A$, then $F_1-F_2$
contains an interval a.s. on $\left\{F_1-F_2 \not= \emptyset
\right\}$. \item If $\gamma_k<1, \gamma_{k+1}<1$ for some $k \in\A$,
then $F_1-F_2$ contains no interval a.s.
\end{enumerate}
\end{theorem}

This result is useful since it can be successfully applied to the
class of correlated fractal percolation, whilst the JSC is never
satisfied for the members of this class. Actually our new
condition can always supersede the JSC.

\begin{lemma}
Suppose that the joint survival distributions $\mu$ and $\lambda$
satisfy the {\rm JSC}. If $\gamma_k >1$ for all $k \in \A$, then
the pair $(\mu,\lambda)$ satisfies the {\rm DGC}.
\end{lemma}

\textbf{Proof.} We take  for the sets $X_k$ and $Y_k$ in
(\ref{gammak}) the marginal supports of $\mu$ and $\lambda$. Then
the JSC implies that (DG0) holds. Since $q_i = 0$ if $i\not\in{\rm
Supp}_m(\lambda)$, and similarly for $p_i$, we have for all
$e\in\A$
\begin{eqnarray*}
\quad\quad\gamma_e({\rm Supp}_m(\mu),{\rm Supp}_m(\lambda)) &=& \sum_{i\in \A} \ind{{\rm Supp}_m(\lambda)}(i)\ind{{\rm Supp}_m(\mu)}(i+e)\\
&\geq & \sum_{i\in\A} q_i p_{i+e} = \gamma_e \geq 2,\nonumber
\end{eqnarray*}
since the number on the left hand side is an integer larger than
1. Therefore  $X_k$ and $Y_k$ certainly satisfy (DG1) and (DG2)
for all $k\in \A$.  Thus $(\mu,\lambda)$  satisfies the DGC.\hfill
$\Box$

\section{Classifying correlated fractal percolation}\label{classify}

With the distributed growth condition at our disposal we can make
an attempt to solve the Palis problem for correlated fractal
percolation. To facilitate our search for sets satisfying the DGC,
we introduce an alternative notation for subsets of the alphabet.
A subset $S$ of the alphabet $\A$ can be represented as a string
of length $M$ with at the $i$th position a zero or a one,
indicating whether or not $i$ is contained in $S$. For
$(m,M,p)$-percolation, all subsets of $\A$ to which is assigned
positive probability correspond to a string consisting of $m$ ones
and $M\!-m$ zeros, where any order of the symbols is allowed. Next
we need the notion of the \emph{cyclic shift operator} $\sigma$.
For any string $X=x_0x_1\dots x_{M-2} x_{M-1}$ we define
\begin{equation}
\sigma (X)=x_1x_2\dots x_{M-1}x_0.
\end{equation}
For the $k$th iterate of $\sigma$ we use the notation $\sigma^k$
and for its inverse $\sigma^{-k}$. Computing $\gamma_k(X,Y)$ can
be done by writing down the two binary strings corresponding to
$\sigma^k(X)$ and $Y$, and then counting in how many positions
both strings have a one (this will be called a
\emph{coincidence}). This procedure is illustrated in
(\ref{illustratie}) for $M=9$, $k=4$ and the sets $X =
\left\{3,5,7,8\right\}$ and $Y = \left\{0,1,6,7\right\}$, where we
abuse notation by \emph{also} writing $X$ for the indicator string
of $X$, and similarly for $Y$ (this will never cause confusion).
\begin{equation}\label{illustratie}
\begin{array}{clllllllllllll}
X&:\quad &0&0&0&1&0&1&0&1&1\\[.3cm]
\sigma^4(X)&:\quad &0&\textbf{1}&0&1&1&0&0&0&1\\
Y&:\quad &1&\textbf{1}&0&0&0&0&1&1&0
\end{array}
\end{equation}

As we see, there is one coincidence, so $\gamma_4(X,Y)=1$.
Checking the DGC boils down to finding binary strings with the
right properties as given in (DG0), (DG1) and (DG2).
\\
\\
Let $X$ and $Y$ be two subsets of the $M$-adic alphabet $\A$
containing $m$ elements in order to satisfy (DG0). Our strategy is
to choose $X$ such that we get a binary string with all ones at
the beginning and $Y$ such that the ones are distributed  evenly
over the string in such a way that at most $m-1$ consecutive zeros
occur. This pattern will lead to fulfillment of requirement (DG1).
If we have sufficient freedom to choose $Y$ within this framework,
then we will also succeed in letting (DG2) be satisfied. The
details of this strategy are filled in in the proof of the lemma
below.

\begin{lemma}\label{Stperc}
For $(m,M,p)$-percolation the following two assertions
hold:
\begin{enumerate}
\item If $m < \sqrt{M}$ or $\displaystyle p<\frac{1}{\sqrt{M}}$,
then $F_1-F_2$ contains no interval a.s.\footnote{Actually, $m <
\sqrt{M}$ implies that $\displaystyle p<1/\sqrt{M}$. Hence the
statement \emph{"If $\displaystyle p<1/\sqrt{M}$, then $F_1-F_2$
contains no interval a.s."} is equivalent to the first assertion of
Lemma \ref{Stperc}. We formulated the lemma in this way to
emphasize what the bounds on $m$ are.} \item If $m \geq \sqrt{M+2}$
and $\displaystyle p>\frac{1}{\sqrt{M}}$, then $F_1-F_2$ contains an
interval a.s. on $\left\{F_1-F_2\not= \emptyset\right\}$.
\end{enumerate}
\end{lemma}

\textbf{Proof.} Suppose that $p< \frac{1}{\sqrt{M}}$, then for all
$k\in\A$ we have
\begin{equation*}
\gamma_k = Mp^2 < M\left(\frac{1}{\sqrt{M}}\right)^2 = 1,
\end{equation*}
and consequently $F_1-F_2$ contains no interval a.s. by Theorem
\ref{StDD09}. If $m<\sqrt{M}$, then $p =
(1-\mu(\emptyset))\frac{m}{M} < \frac{1}{\sqrt{M}}$ and
consequently the same argument is applicable, completing the proof
of the first part of Lemma \ref{Stperc}.
\\
\\
For the proof of the second assertion, assume that $m \geq
\sqrt{M+2}$ and define $X,Y'\subseteq\A$ by their strings
\begin{eqnarray*}\label{defsets}
X&=&1^{m}\,0^{M\!-m}\nonumber\\
Y'\!&=&R\,[1\,0^{m-1}]^q,
\end{eqnarray*}
where $q=\lfloor M/m \rfloor$,$R$ is a left substring of
$1\,0^{m-2}$ ($R$ is empty when $m$ divides $M$), and
$[1\,0^{m-1}]^q$ denotes the string $1\,0^{m-1}$, $q$ times
repeated. Ignoring the trivial case $M=m=2$ we obtain from $m \geq
\sqrt{M+2}$  that we may assume $m\geq 3$.
\\
\\
Since $Y'$ does not contain $m$ consecutive zeros (also
cyclically), whereas $X$ begins with $m$ consecutive 1's, we must
have
\begin{equation*}\label{RG2a}
\gamma_e(X,Y') \geq 1 \quad\quad\mbox{for\ } e = 0,1,\ldots,M-1.
\end{equation*}

So $X$ and $Y'$ satisfy (DG1). The set $X$ contains $m$ elements,
which means that $\mu(X)>0$.
\\
Note that $q=\lfloor M/m \rfloor$ can not exceed $m-1$, since that
would imply $m\leq \sqrt{M}$.
\\
\\
{\it Case 1: $q\leq m-2$ or  $R$ is empty.} \\
Then $Y'$ contains at most $m-1$ ones. In order to obtain (DG2),
we construct $Y''$ from $Y'$ by putting a one in the second
position (if there is a zero)---note that $X$ and $Y''$ will then
certainly still satisfy (DG1). Moreover, we now have
\begin{equation*}
\gamma_0(X,Y'')\geq 2,\quad\quad \gamma_1(X,Y'')\geq 2,
\end{equation*}
since $m\geq 3$. Finally $Y$ is obtained by adding 1's to $Y''$
(if necessary) till $Y$ contains $m$ ones---and thus $\mu(Y)>0$.
As an illustration for $M=7$ and $m=4$, $X$ is given by $1111000$
and (writing $\gamma_k(\cdot)$ for $\gamma_k(X,\cdot)$):
\begin{equation*}\label{illustratie2}
\begin{array}{l|c|c|c|c}
\;\cdot &\mbox{String}& \mu(\cdot)>0& \gamma_e(\cdot)\geq 1\quad \forall e\in\A &\gamma_0(\cdot),\gamma_1(\cdot)\geq 2 \\
\hline\\[-.42cm]
Y':\quad &1001000&\mbox{no}&\mbox{yes}&\mbox{no}\\
Y'':\quad &1101000&\mbox{no}&\mbox{yes}&\mbox{yes}\\
Y:\quad &1111000&\mbox{yes}&\mbox{yes}&\mbox{yes}\\
\end{array}
\end{equation*}

Now we have found $X_0:=X$ and $Y_0:=Y$ satisfying (DG0), (DG1)
and (DG2) for $k=0$. By observing that
\begin{equation}\label{eq:symm}
\gamma_k(X,\sigma^{k}Y) = \gamma_0(X,Y);\quad \gamma_{k+1}(X,\sigma^{k}Y) = \gamma_1(X,Y),
\end{equation}
it follows that the DGC holds for any $k\in \A$ if we take $X_k = X$ en $Y_k=\sigma^{k}Y$.
\\
\\
{\it Case 2: $q = m-1$ and $R\ne\emptyset$.} \\  Since $m\ge \sqrt{M+2}$, we have
(with $r$ the length of $R$)
$$M=m^2-m+r\ge M+2-m+r,$$
so $r\le m-2$, implying that $R$ does not contain more than $m-3$
zero's. This gives that $\gamma_0(X,Y)\geq 2$ and
$\gamma_1(X,Y)\geq 2$. Now again we can take $X_k = X$ en
$Y_k=\sigma^{k}Y$. Summarizing, for all cases of correlated
fractal percolation in part (2) we have shown that (DG0), (DG1)
and (DG2) hold. We conclude that the DGC is satisfied.

Moreover, for all $k\in\A$ we find
\begin{equation*}
\gamma_k = \sum_{j=0}^{M-1} p_j p_{j+k} = Mp^2> M\left(\frac{1}{\sqrt{M}}\right)^2 = 1,
\end{equation*}
and therefore, by Theorem \ref{StDD09}, $F_1-F_2$ contains an
interval a.s. on $\left\{F_1-F_2\not= \emptyset\right\}$.\hfill $\Box$
\\
\\
Lemma \ref{Stperc} still gives no conclusive answer for some
combinations of $m$ and $M$ when $p>1/\sqrt{M}$, namely, those
where $m=\sqrt{M+1}$. By having a look at the $2^{\rm nd}$ order
sets for $(m,M,p)$-percolation this can be resolved.

\begin{lemma}\label{Stperc2}
Consider $(m,M,p)$-percolation. If  $\displaystyle
p>\frac{1}{\sqrt{M}}$ and \begin{equation}\label{sqrtM} m =
\sqrt{M+1},
\end{equation}
 then $F_1-F_2$ contains an
interval a.s. on $\left\{F_1-F_2\not= \emptyset\right\}$.
\end{lemma}

\textbf{Proof.} First we have a look at the shape of the binary
strings corresponding to $2^{\rm nd}$ order sets to which is
assigned positive probability by the $2^{\rm nd}$ order joint
survival distribution $\mu^{(2)}$ of correlated fractal
$(m,M,p)$-percolation. Such a string has length $M^2$. It should
be regarded as consisting of $M$ blocks of length $M$. Each of
these blocks contains either exclusively zeros, or it contains
$M-m$ zeros and $m$ ones. Blocks of the latter kind occur exactly
$m$ times. Positions in the binary string can be identified with
numbers in $\A^{(2)}$: an $M^2$-adic number represented by
$\underline{k}_2 = k_1k_2$ corresponds to the $(k_2+1)$th position
in the $(k_1+1)$th block.
\\
\\
Note that (\ref{sqrtM}) implies that $M-m(m-1)=m-1$ and $\lfloor
M/m \rfloor=m-1$. This means that the two strings $X$ and $Y'$
defined in the proof of Lemma \ref{Stperc} are now equal to (we
omit from now on the prime on $Y$)
\begin{eqnarray*}\label{defsets2}
X&=&1^{m}\,0^{M\!-m},\nonumber\\
Y&=&[1\,0^{m-2}]\,[1\,0^{m-1}]^{m-1}.
\end{eqnarray*}

The basic idea of the proof is to replace the $0$'s in  these two
strings by blocks $0^M$, and the $1$'s by blocks similar to $X$ or
$Y$ to obtain for all $\underline{k}_2\in\A^{(2)}$ the order 2
strings $X_{\underline{k}_2}^{(2)}$ and
$Y_{\underline{k}_2}^{(2)}$ which will satisfy (DG1) and
(DG2)---note that by construction (DG0) is then obviously
satisfied.

Actually we will replace all the $m$ $1$'s in $X$ by the string
$Y$. Replacing additionally the $M\!-m$ $0$'s by blocks $0^M$ we
obtain $X_{\underline{k}_2}^{(2)}$ independent of
$\underline{k}_2$, and hence we will denote it by $X^{(2)}$.
\\
\\
The definition of $Y_{\underline{k}_2}^{(2)}$ is slightly more
involved. We first restrict ourselves to the case $k_1=0$ and
define:
\begin{equation*}
Y_{0k_2}^{(2)} :=
\sigma^{Ms}\left([\sigma^{k_2}(X)\,0^{(m-2)M}]\,[\sigma^{k_2}(X)\,0^{(m-1)M}]^{m-1}\right),
\end{equation*}
where $s$ is given by
\begin{equation*}
s:=\left\{\begin{array}{ll}
0 &\quad\quad \mbox{if\ } 0\leq k_2 \leq m-2,\\
1 &\quad\quad \mbox{if\ } m-1\leq k_2 \leq M-1.
\end{array}\right.
\end{equation*}
So the $m$ $1$'s in $Y$ are replaced by shifted versions of $X$
and $0$'s by blocks $0^M$ and finally an additional shift over $M$
positions is applied on the complete string if $k_2$ is at least
$m-1$.

\begin{example}\label{M8m3}
Let $M=8$ and $m=3$. Then
$$X=11100000 \quad {\rm and} \quad Y=10100100.$$
Writing $O=0^8$ and $s= \mathbf{1}_{\left\{n:n\geq
2\right\}}(k_2)$, we have for $0k_2\in \A^{(2)}$
\begin{equation*}
\begin{array}{rc@{\hspace{6mm}} c\SHa c\SHa c\SHa c\SHa c\SHa c\SHa c\SHa c\SHa c}
X^{(2)}                     &=& &\lp Y\rp        &\lp Y\rp &\lp Y\rp        & \lp O\rp & \lp O\rp & \lp O\rp        & \lp O\rp & \lp O\rp \\
Y^{(2)}_{0k_2}   &=& \sigma^{8s}\bigl(&\sigma^{k_2}(X) &\lp O\rp &\sigma^{k_2}(X) & \lp O\rp & \lp O\rp & \sigma^{k_2}(X) & \lp O\rp & \lp O \bigr).\phantom{.}
\end{array}
\end{equation*}
\end{example}

Suppose that $X^{(2)}$ and $Y^{(2)}_{0k_2}$ satisfy the DGC. Then
it is easy to construct sets $X^{(2)}$ and $Y^{(2)}_{k_1k_2}$
satisfying requirements (DG1) and (DG2) for other values of $k_1$.
First observe that all shifted versions of $X^{(2)}$ and
$Y^{(2)}_{0k_2}$ still satisfy (DG1). Furthermore we use the fact
that
\begin{eqnarray*}
\gamma^{(2)}_{\underline{k}_2}\big(X^{(2)},\sigma^{k_1M}(Y^{(2)}_{0k_2})\big)&=&
\gamma^{(2)}_{0k_2}(X^{(2)},Y^{(2)}_{0k_2})\geq 2,\\
\gamma^{(2)}_{\underline{k}_2+1}\big(X^{(2)},\sigma^{k_1M}(Y^{(2)}_{0k_2})\big)&=&
\gamma^{(2)}_{(0k_2)+1}(X^{(2)},Y^{(2)}_{0k_2})\geq 2.
\end{eqnarray*}
Now it follows that we can choose $Y^{(2)}_{\underline{k}_2} = \sigma^{k_1M}(Y^{(2)}_{0k_2})$.
\\
\\
To complete the proof, it suffices to check that the sets
$X^{(2)}$ and $Y^{(2)}_{0k_2}$ satisfy requirements (DG1) and
(DG2) of the DGC. Therefore, we consider
the correlation coefficients
$\gamma^{(2)}_{\underline{e}_2}(X^{(2)},Y^{(2)}_{0k_2})$ where
$\underline{e}_2=e_1e_2\in\A^{(2)}$. We will focus first on the
`coarse' structure, i.e. on those correlation coefficients for
which $e_2=0$. Here we will always have a string $\sigma^{k_2}(X)$
in $Y^{(2)}_{0k_2}$ coinciding with a string $Y$ in $X^{(2)}$ for
the same reason that we always have a coincidence at level $1$.
This implies that we also always have a string $\sigma^{k_2}(X)$
in $Y^{(2)}_{0k_2}$ coinciding with a zero string of length $M$ in $X^{(2)}$
which is followed (cyclically) by a string $Y$.

It follows that if we will shift on the `fine' level by varying
$e_2$, then in all cases we are in the same situation of one
$\sigma^{k_2}(X)$ block `entering' an $Y$ block, and one
$\sigma^{k_2}(X)$ `leaving' an $Y$ block. Thus we get the same
coincidences as in the case where $\sigma^{k_2}(X)$ and $Y$ are
compared cyclically, and therefore the second order correlation
coefficients can be related to the first order correlation
coefficients $\gamma_e(\sigma^{k_2}(X),Y)$:
\begin{equation}\label{relation}
\gamma^{(2)}_{\underline{e}_2}(X^{(2)},Y^{(2)}_{0k_2}) \geq
\gamma_{e_2}(Y,\sigma^{k_2}(X)) \geq 1
\end{equation}
for all $\underline{e}_2=e_1e_2\in\A^{(2)}$. As we see, (DG1) holds for all $\underline{e}_2\in\A^{(2)}$.
\\
\\
Now we turn to (DG2). If $e_2=k_2$, then in (\ref{relation}) we
even have by equation (\ref{eq:symm}) that
\begin{equation*}
\gamma^{(2)}_{\underline{e}_2}(X^{(2)},Y^{(2)}_{0k_2}) \geq
\gamma_{e_2}(Y,\sigma^{k_2}(X))  = \gamma_0(Y,X) = 2,
\end{equation*}
which means that
\begin{equation}\label{RG2firstpart}
\gamma^{(2)}_{0k_2}(X^{(2)},Y^{(2)}_{0k_2}) \geq 2.
\end{equation}
We still have to check that also
$\gamma^{(2)}_{(0k_2)+1}(X^{(2)},Y^{(2)}_{0k_2}) \geq 2$. First we
concentrate on the case where both the first and the last
$Y$-block in $X^{(2)}$ coincide with a $\sigma^{k_2}(X)$ block in
$Y^{(2)}_{0k_2}$. To illustrate this in the terms of Example \ref{M8m3}, we
have:
\begin{equation*}
\begin{array}{rc@{\hspace{6mm}} c\SHa c\SHa c\SHa c\SHa c\SHa c\SHa c\SHa c}
X^{(2)}                     &=& \lp Y\rp        &\lp Y\rp &\lp Y\rp        & \lp O\rp & \lp O\rp & \lp O\rp        & \lp O\rp & \lp O\rp \\
Y^{(2)}_{0k_2}   &=& \sigma^{k_2}(X) &\lp O\rp &\sigma^{k_2}(X) &
\lp O\rp & \lp O\rp & \sigma^{k_2}(X) & \lp O\rp & \lp
O.\phantom{..}
\end{array}
\end{equation*}
Keeping $k_1$ fixed to zero and varying $k_2$, the structure of
coincidences we obtain will look like:

\begin{tiny}
\begin{equation*}\label{illustratie4}
\begin{array}{l\SH l\SH l\SH l\SH l\SH l\SH l\SH l\SH l\SH l\SH l\SH l\SH l\SH l\SH l\SH l\SH l\SH l\SH l\SH l\SH l\SH l\SH l@{\hspace{5mm}} l@{\hspace{4mm}} l}
\multicolumn{5}{l}{YY:}&|1&0 &\textbf{1} & 0 & 0 & 1 & 0 & 0|& 1 & 0 & \textbf{1} & 0 & 0 & 1 & 0 & 0|&&&k_2\!\downarrow&s\!\downarrow\\
\hline\\[-.2cm]
\multicolumn{6}{l}{\sigma^{k_2}(X)\sigma^{k_2}(X):}&|1 & 1 & 1& 0& 0& 0& 0& 0|& 1 & 1 & 1 & 0 & 0 & 0 & 0 & 0|&    &0  &0 \\
                               &   &   &   &   &   &   &|1 & 1& 0& 0& 0& 0& 0 & 1|& 1 & 1 & 0 & 0 & 0 & 0 & 0 & 0| &1  &0 \\
                             |0& 0& 0& 0& 0& 0& 1& 1&|1 & 0& 0& 0& 0& 0 & 1 & 1|&    &    &    &    &    &    &    &2  &1 \\
                               &|0 & 0& 0& 0& 0& 1& 1& 1&\!|0 & 0& 0& 0& 0 & 1 & 1 & 1|&  &    &    &    &    &    &3  &1 \\
                               &   &|0 & 0& 0& 0& 1& 1& 1& 0&|0 & 0& 0& 0 & 1 & 1 & 1 & 0|&    &    &    &    &    &4  &1 \\
                               &   &   &|0 & 0& 0& 1& 1& 1& 0& 0&|0 & 0& 0 & 1 & 1 & 1 & 0 & 0|&    &    &    &    &5  &1 \\
                               &   &   &   &|0 & 0& 1& 1& 1& 0& 0& 0&|0 & 0 & 1 & 1 & 1 & 0 & 0 & 0|&    &    &    &6  &1 \\
                               &   &   &   &   &|0 & 1& 1& 1& 0& 0& 0& 0&|0  & 1 & 1 & 1 & 0 & 0 & 0 & 0|&    &    &7  &1 \\
\end{array}
\end{equation*}
\end{tiny}

Each line in the table corresponds to a value of $k_2$ and
displays the string $\sigma^{k_2}(X)\sigma^{k_2}(X)$. This string
is moved over $k_2+1$ positions to the right, since we are
interested in $\gamma^{(2)}_{(0k_2)+1}(X^{(2)},Y^{(2)}_{0k_2})$.
Then, for each value of $k_2$ the corresponding value of $s$
(being either $0$ or $1$) is computed. If $s=1$, then the string
is moved over $M=8$ positions back to the left. By construction,
the number of coincidences of $YY$ with the $k_2$-line in the
table is a lower bound for
$\gamma^{(2)}_{(0k_2)+1}(X^{(2)},Y^{(2)}_{0k_2})$. In each of the
lines of the table, we have coincidences with both bold ones in
$YY$. Therefore,
$$\gamma^{(2)}_{(0k_2)+1}(X^{(2)},Y^{(2)}_{0k_2})\geq 2.$$

Combining this with (\ref{RG2firstpart}), we see that (DG2) holds. Adapting this argument for other values of $M$ and $m$ is
straightforward.
\\
\\
As we have seen in the proof of the previous lemma, it is possible to find sufficient independent left and right triangles. Therefore, we have completed our proof that the distributed growth condition is satisfied. We also already saw $p>1/\sqrt{M}$ implies that $\gamma>1$, and hence we can use Theorem \ref{StDD09} to finish the proof of Lemma \ref{Stperc2}.\hfill $\Box$

\begin{theorem}\label{corrfrac-class}
For correlated fractal $(m,M,p)$-percolation we have
\begin{enumerate}
\item If $\gamma>1$ then $F_1-F_2$
contains an interval a.s. on $\left\{F_1-F_2 \not= \emptyset
\right\}$. \item If $\gamma\le 1$, then
$F_1-F_2$ contains no interval a.s.
\end{enumerate}
\end{theorem}
\textbf{Proof.}  This result is the combination of Lemma
\ref{Stperc}, Lemma \ref{Stperc2} and Proposition
\ref{crit}.\hfill $\Box$

We remark here that since these results will also hold if we
merely require that all sets with $m$ elements have positive
probability to occur, the theorem will also be true in this more
general case.

\section{The lower spectral radius in the symmetric case}

In this section we show that the distributed growth condition
propagates to higher order Cantor sets. As a consequence, the
spectral radius characterization obtained in \cite{DK08} can be
extended to joint survival distributions satisfying the DGC.

\begin{lemma}\label{lemma_propagation}
(Propagation of the distributed growth condition to higher orders)
Suppose the pair of joint survival distributions $(\mu,\lambda)$
satisfies the {\rm{DGC}}. Then for all $n\geq 1$, the pair of
$n$th order joint survival distributions
$(\mu^{(n)},\lambda^{(n)})$ satisfies the {\rm{DGC}}.
\end{lemma}

\textbf{Proof.} Choose a  string $\underline{k}_n\in\A^{(n)}$,
which we write as $\underline{k}_n = kk_2\ldots k_n $, with
$k\in\A$ and $k_2\dots k_n\in\A^{(n-1)}$. We check that we can
find $n$th order sets satisfying the DGC for this
$\underline{k}_n$. Since the pair $(\mu,\lambda)$ satisfies the
DGC, there exist first order sets $X_{k}, Y_{k}\subseteq \A$
satisfying (DG0), (DG1) and (DG2). Define
\begin{eqnarray*}
X_{k}^{(n)} & := & \left\{ \underline{l}_n = l_1\ldots l_n\in\A^{(n)}:l_j\in X_{k}\mbox{\ for\ all\ } j = 1,\ldots,n \right\},\nonumber\\
Y_{k}^{(n)} & := & \left\{ \underline{l}_n = l_1\ldots l_n\in\A^{(n)}:l_j\in Y_{k}\mbox{\ for\ all\ } j = 1,\ldots,n \right\}.
\end{eqnarray*}

Obviously, $\mu^{(n)}(X_{k}^{(n)}) > 0$ and
$\lambda^{(n)}(Y_{k}^{(n)}) > 0$. Define a new pair of $n$th order
joint survival distributions by $\mu_{k}^{(n)}(X_{k}^{(n)}) =
\lambda_{k}^{(n)}(Y_{k}^{(n)}) = 1$. Also define a first order
deterministic pair of joint survival distributions by
$\mu_{k}(X_{k}) = \lambda_{k}(Y_{k}) = 1$. The expectation
matrices belonging to these $n$th order survival distributions are
related to those belonging to the first order survival
distributions by
\begin{equation*}
\mathcal{M}_{k}^{(n)}(\underline{k}_n) =
\mathcal{M}_{k}(\underline{k}_n) =
\mathcal{M}_{k}(k)\mathcal{M}_{k}(k_2)\ldots\mathcal{M}_{k}(k_n).
\end{equation*}
Using that $X_{k}$ and $Y_{k}$ satisfy  (DG1) and (DG2), and that
the columns sums of the expectation matrices are equal to the
correlation coefficients, we obtain that
\begin{eqnarray*}
[1\;\; 1]\mathcal{M}_{k}^{(n)}(\underline{k}_n) = [1\;\; 1]
\mathcal{M}_{k}(k)\prod_{j=2}^{n}\mathcal{M}_{k}(k_j) \geq [2\;\;
2]\prod_{j=2}^{n}\mathcal{M}_{k}(k_j)\geq [2\;\; 2]\nonumber
\end{eqnarray*}
elementwise, which means that $Z_{k}^{(n);L}(\underline{k}_n)\geq
2$ and $Z_{k}^{(n);R}(\underline{k}_n)\geq 2$, or equivalently
\begin{equation*}
\gamma_{\underline{k}_n}(X_{k}^{(n)},Y_{k}^{(n)})\geq 2;\quad\quad
\gamma_{\underline{k}_n+1}(X_{k}^{(n)},Y_{k}^{(n)})\geq 2.
\end{equation*}
Similarly $\gamma_{\underline{l}_n}(X_{k}^{(n)},Y_{k}^{(n)})\geq
1$ for all $\underline{l}_n\in\A^{(n)}$. It follows that the pair
$(\mu^{(n)},\lambda^{(n)})$ satisfies the  DGC. \hfill $\Box$
\\
\\
This propagation property leads to the theorem below. The lower
spectral radius $\underline{\rho}(\Sigma)$ of a set $\Sigma$ of
square matrices is defined by
\begin{equation*}
\underline{\rho}(\Sigma) := \liminf_{n\rightarrow\infty} \min_{A_1,\ldots,A_n\in\Sigma}||A_1\ldots A_n||^{1/n},
\end{equation*}
for some matrix norm $||\cdot||$. For two $M$-adic random Cantor
sets, let $\Sigma_M$ be the corresponding collection of
expectation matrices
\begin{equation}
\Sigma_M := \left\{\mathcal{M}(0),\ldots,\mathcal{M}(M-1)\right\}.
\end{equation}
Then we obtain the following result:

\begin{theorem}\label{St:spectral}
Consider the algebraic difference $F_1-F_2$ between two $M$-adic
independent random Cantor sets $F_1$ and $F_2$ with the same joint
survival distribution satisfying the distributed growth condition.
\begin{enumerate}
\item If $\underline{\rho}(\Sigma_M)>1$, then $F_1-F_2$ contains no interval a.s. on $\left\{F_1-F_2\not=\emptyset\right\}$.
\item If $\underline{\rho}(\Sigma_M)<1$, then $F_1-F_2$ contains no interval a.s.
\end{enumerate}
\end{theorem}
\textbf{Proof.} The proof is basically the same as the proof of
Theorem 6.1 in \cite{DK08}. There is a difference in the fact that
here we do not require irreducibility explicitly. From the
symmetry $\mu=\lambda$ it follows that $m_e=m_{-e}$ for all
$e\in\A\cup -\A$. Now, since the DGC holds, we get the
irreducibility for free.

After derivation of the same statements concerning the $n$th order
correlation coefficients as in \cite{DK08}, we apply our Theorem
\ref{StDD09}. This is justified by the fact that the DGC
propagates to higher orders, as was shown in Lemma
\ref{lemma_propagation}. \hfill$\Box$

\section{Final remarks}

We have solved the problem of the Palis conjecture for correlated
fractal percolation (Theorem~\ref{corrfrac-class}), even for the
critical case. For this we introduced a new growth condition,
which as a bonus gives a more general, and a more simple proof of
the basic theorem (Theorem~\ref{StDD09}). It is more simple since
we do not need the combinatorial `color lemma' of \cite{DS08} and
\cite{DK08}, nor the irreducibility condition of \cite{DK08}. The
counterexample of \cite{DK08} (where the spectral radius is larger
than 1, but still there is no interval in the algebraic
difference) is now explained by the fact that $\gamma=1$ in that
example. In view of Proposition~\ref{crit}, this makes us
conjecture that \emph{in general} the algebraic difference
$F_1-F_2$ will not contain an interval if $\dim_{\rm
H}(F_1)+\dim_{\rm H}(F_2)=1$ (except for deterministic sets).

\end{document}